\documentclass[a4paper,12pt]{amsart}

\usepackage{amsfonts,amssymb,amscd,amsmath,latexsym,amsbsy,enumerate,stmaryrd,a4wide,verbatim,bm}
\usepackage{color}
\usepackage{hyperref}

\theoremstyle{plain}
\newtheorem{theorem}{Theorem}[section]

\newtheorem{lemma}[theorem]{Lemma}

\theoremstyle{remark}
\newtheorem{remark}[theorem]{Remark}
\numberwithin{equation}{section}

\newcommand{\N}{\mathbb N}
\newcommand{\C}{\mathbb C}

\newcommand{\B}{\mathbb B}
\newcommand{\M}{\mathbb M}
\newcommand{\al}{\alpha}
\newcommand{\be}{\beta}

\newcommand{\si}{\sigma}

\newcommand{\La}{\Lambda}

\newcommand{\m}{\mathbf{m}}

\newcommand{\bn}{\mathbf{n}}
\newcommand{\bb}{\mathbf{b}}
\newcommand{\bc}{\mathbf{c}}
\newcommand{\bk}{\mathbf{k}}

\newcommand{\bp}{\mathbf{p}}
\newcommand{\bt}{\mathbf{t}}
\newcommand{\bv}{\mathbf{v}}

\newcommand{\bx}{\mathbf{x}}

\newcommand{\bz}{\mathbf{z}}

\newcommand{\cA}{\mathcal A}
\newcommand{\cB}{\mathcal B}

\newcommand{\te}{\tilde e}

\newcommand{\tensor}{\otimes}

\newcommand{\rFs}[5]{\,_{#1}F_{#2} \left( \genfrac{.}{.}{0pt}{}{#3}{#4}
	\ ;#5 \right)}

\newcommand{\su}{\mathfrak{su}}

\newcommand{\SU}{\mathrm{SU}}

\newcommand{\diag}[1]{\operatorname{diag}(#1)}
\newcommand*\diff{\mathop{}\!\mathrm{d}}
\newcommand{\dv}{\diff v_{\alpha}}

\begin{document}
	\title[Meixner polynomials related to holomorphic discrete series]{Multivariate Meixner polynomials related to holomorphic discrete series representations of $\SU(1,d)$}
	\author{Wolter Groenevelt}
	\address{Technische Universiteit Delft, DIAM, PO Box 5031,
		2600 GA Delft, the Netherlands}
	\email{w.g.m.groenevelt@tudelft.nl}
	\author{Joop Vermeulen}
	\address{Technische Universiteit Eindhoven, Department of Mathematics and Computer Science, PO Box 513, 5600 MB Eindhoven, the Netherlands} 
	\email{j.j.m.vermeulen@tue.nl}
	\maketitle

\begin{abstract}
	We show that Griffiths' multivariate Meixner polynomials occur as matrix coefficients of holomorphic discrete series representations of the group $\SU(1,d)$. Using this interpretation we derive several fundamental properties of the multivariate Meixner polynomials, such as orthogonality relations and difference equations. Furthermore, we also show that matrix coefficients for specific group elements lead to degenerate versions of the multivariate Meixner polynomials and their properties.
\end{abstract}

\section{Introduction}
In this paper we study Griffiths' \cite{Gr75} multivariate generalization of the Meixner polynomials. The (univariate) Meixner polynomials $M_n$, first studied by Meixner in \cite{Mei}, can be  defined through their generating function by
\[
\left(1-\tfrac{t}{c}\right)^x (1-t)^{-x-\beta} = \sum_{n=0}^\infty \frac{ (\be)_n }{n!} M_n(x;\be,c)t^n.
\]
These polynomials have many nice properties, e.g.~they are of hypergeometric type,
\[
M_n(x;\be,c) = \rFs{2}{1}{-n,-x}{\be}{1-\frac1c},
\]
and they are orthogonal with respect to the negative binomial distribution, 
\[
\sum_{x=0}^\infty M_n(x;\be,c) M_{n'}(x;\be,c) \frac{(\be)_x}{x!}c^x (1-c)^{\be} = 0, \qquad \text{if} \quad n \neq n',
\]
where we assume $\be>0$ and $c \in (0,1)$. Here we use standard notation for shifted factorials and hypergeometric functions, see e.g.~\cite{AAR}.

It is well known that the Meixner polynomials are related to the Lie group $\SU(1,1)$, which is the group of complex $2\times 2$-matrices $\left( 
\begin{smallmatrix} a & b \\ \overline{b} & \overline{a} \end{smallmatrix}
\right)$ with $|a|^2-|b|^2=1$. To be more precise, the Meixner polynomials occur as matrix coefficients of the holomorphic discrete series representations defined as follows. For a positive integer $\si$, the Bergman space $\mathcal A_\si$ is the Hilbert space of holomorphic functions on the complex unit disc $\mathbb D$ with inner product
\[
\langle f_1,f_2\rangle = \int_{\mathbb D} f_1(z)\overline{f_2(z)} (1-|z|^2)^\si \, dz.
\]
The holomorpic discrete series are the irreducible unitary representations on $\mathcal A_\si$ given by
\[
[\pi(g) f] (z) = \frac{1}{(a+\overline{b}z)^\si} f \left( \frac{b+\overline{a}z}{a+\overline{b}z} \right), \qquad g = \begin{pmatrix} a & b \\ \overline{b} & \overline{a} \end{pmatrix} \in \SU(1,1).
\]
The monomials $z^n$, $n=0,1,2,\ldots$, form an orthogonal basis for $\mathcal A_\si$. Let $\pi_{m,n}(g)$ be the coefficients in the expansion of $\pi(g)z^n$ in terms of this basis, i.e.
\[
(b+\overline{a}z)^n (a+\overline{b}z)^{-\si-n}= \sum_{m=0}^\infty \pi_{n,m}(g) z^m.
\] 
Comparing this to the generating function of the Meixner polynomials, it follows that the coefficients $\pi_{m,n}(g)$ are multiples of the Meixner polynomials $M_m(n;\si,|\tfrac{b}{a}|^2)$. 
From this representation theoretic interpretation of the Meixner polynomials several useful properties of the polynomials can easily be obtained: e.g.~orthogonality relations, the three-term recurrence relation and the second order difference equation.

In \cite{Gr75}, Griffiths introduced multivariate  Meixner polynomials as orthogonal polynomials with respect to the negative multinomial distribution. As such, they are closely related to the multivariate Krawtchouk polynomials \cite{Gr} which are orthogonal on a finite set with respect to the usual multinomial distribution. Iliev \cite{Il} showed that the multivariate Krawtchouk polynomials occur as matrix coefficients for finite dimensional representations of $\mathrm{SL}(d,\C)$. This leads, for example, to bispectrality of the Krawtchouk polynomials. In a follow-up paper \cite{Il12} it is shown,  without references to representation theory, that the multivariate Meixner polynomials have many properties that resemble those of the Krawtchouk polynomials. A representation theoretic interpretation of Griffiths' bivariate Meixner polynomials is obtained by Genest, Miki, Vinet and Zhedanov \cite{GMVZ} by showing they appear as matrix coefficients of $\mathrm{SO}(2,1)$ representations on oscillator states. In the same paper, it is also indicated how the general multivariate Meixner polynomials arise similarly in the representation theory of $\mathrm{SO}(d,1)$. Furthermore, in \cite{GGLV} it is shown that the bivariate Meixner polynomials occur as wave functions for a two-dimensional quantum oscillator. 

In the present paper we show an alternative way to study the multivariate Meixner polynomials by using holomorphic discrete series representations of $\SU(1,d)$, similar to the above described representation theoretic interpretation of the univariate Meixner polynomials. See also e.g.~\cite[Section 6.8]{VK} for the interpretation of the univariate Meixner polynomials in $\SU(1,1)$ representations. We expect that the results can be generalized to the more general case of holomorphic discrete series representations of $\SU(n,m)$ (for $n,m \geq 2$). This will be the topic of a future paper.

The organization of the paper is as follows. First in Section \ref{sec:SU(1,d)} we recall the holomorphic discrete series representation of $\SU(1,d)$. Then, in Section \ref{sec:Meixner polynomials} we recall Griffiths' definition of the multivariate Meixner polynomials and show that the matrix coefficients of the holomorphic discrete series corresponding to generic $g \in \SU(1,d)$ can be expressed in terms of these Meixner polynomials. This immediately leads to several properties, such as orthogonality with respect to the negative multinomial distribution, of the Meixner polynomials. In Section \ref{sec:Lie algebra}, we consider the corresponding Lie algebra representations and derive difference equations for the multivariate Meixner polynomials. Finally, in Section \ref{sec:degenerate Meixner} we consider degenerate versions of the multivariate Meixner polynomials, which correspond to the matrix coefficients for specific elements $g \in \SU(1,d)$.

\subsection{Notations}
For $\bx=(x_1,\ldots,x_d) \in \C^d$, we define
\[
\begin{gathered}
\|\bx\|=\sqrt{|x_1|^2+\ldots+|x_d|^2},\\
|\bx| = x_1+\ldots+x_d,\\
\overline{\bx} = (\overline{x_1},\ldots,\overline{x_d}).
\end{gathered}
\]
We often consider elements in $\C^d$ as column vectors, which will be clear from the context. 
For $\bn=(n_1,\ldots,n_d) \in \N_0^d$ and $\bx = (x_1,\ldots,x_d) \in \C^d$, we set
\[
\bn! = n_1!\cdots n_d!, \quad \bx^{\pm\bn} = x_1^{\pm n_1}\cdots x_d^{\pm n_d}.
\]
We denote by $\M_d$ the set of complex $d\times d$-matrices. For $A \in \M_d$, $A^t$ is the transpose of $A$ and $A^\dagger$ the conjugate transpose.

\section{The holomorphic discrete series representations of $\SU(1,d)$}  \label{sec:SU(1,d)}
$\mathrm{SL}(d+1;\C)$ is the group of complex $(d+1)\times (d+1)$-matrices of determinant 1. $\SU(1,d)$ is the subgroup of $\mathrm{SL}(d+1;\C)$ preserving the hermitian form associated with the matrix 
\[J = \diag{1,-1,-1,\dots,-1},\]
that is, a matrix $g \in \mathrm{SL}(d+1;\C)$ is in $\SU(1,d)$ if and only if the following equation holds:
\begin{equation}\label{eq:3:defining SU(1,d)}
	g^\dagger J g = J.
\end{equation}
Throughout this paper it will be convenient to write $g\in \SU(1,d)$ in the form 
\[
g = \begin{pmatrix} a & \mathbf{b}^t\\\mathbf{c} & D \end{pmatrix},
\]
where $a\in \C$, $\mathbf{b},\mathbf{c}\in\C^d$ and $D \in \M_d$. We will sometimes use notations such as $a=a(g)$ to stress the dependence on $g$. From the defining Equation \eqref{eq:3:defining SU(1,d)}, it follows that the inverse of the matrix $g \in \SU(1,d)$ is given by
\begin{equation} \label{eq:g inverse}
	g^{-1} = Jg^{\dagger}J = \begin{pmatrix} \overline{a} & -\mathbf{c}^{\dagger} \\ -\overline{\mathbf{b}} & D^{\dagger} \end{pmatrix},
\end{equation}
which implies identities such as 
\begin{equation} \label{eq:identities abcD}
\begin{gathered}
|a|^2-\|\bb\|^2=|a|^2-\|\bc\|^2=1, \\ D^\dagger D =I_d+ \overline{\bb}\bb^t, \quad D D^\dagger = I_d+\bc \bc^\dagger, \\ \overline{a}\bb^t=\bc^\dagger D, \qquad a\bc^\dagger = \bb^tD^\dagger,
\end{gathered}
\end{equation}
where $I_d$ is the identity matrix in $\M_d$. 

$\SU(1,d)$ has a family of representations called the holomorphic discrete series, on the weighted Bergman space $\cA_\alpha$ that we now introduce, see e.g.~\cite[Chapter VI]{Kn}. Let $\alpha>-1$. We define $\dv$ to be the weighted Lebesgue measure on the open unit ball $\B_d = \{\bz\in \C^d \,|\, \|\bz\|< 1\}$ given by 
\begin{equation}
	\dv = c_\alpha (1-\|\bz\|^2)^{\alpha}\diff v,
\end{equation}
with $\diff v$ the standard volume measure on $\B_d$ and $c_{\alpha}$ is the normalizing constant so that $v_{\alpha}(\B_d)=1$. A direct calculation shows that 
\[
c_\alpha =  \frac{(\alpha+1)_d}{d!}.
\]
The Bergman space $\cA_\alpha$ is the space of holomorphic functions in $L^{2}(\B_d,\dv)$. $\cA_\alpha$ is a Hilbert space with inner product
\[
	\langle f,g \rangle = \int_{\B_d}f(\bz)\overline{g(\bz)}\dv, \qquad f,g\in\cA_\alpha.
\]
An orthonormal basis for $\cA_\alpha$ is given by the monomials
\[
	e_\m(\bz) = \sqrt{\frac{(\alpha+d+1)_{|\m|}}{\m!}} \bz^\m, \qquad \m \in \N_0^d,
\]
see e.g.~Lemma 1.11 and Proposition 2.6 in \cite{Zhu}.

Now we are ready to define the representation of $\SU(1,d)$ we are interested in this paper. From here on, we assume $\alpha \in \N_0$ and we set
\[
\sigma = \alpha+d+1.
\]
Then, $\pi^\sigma$ given by
\begin{equation}\label{eq:3:rep definition}
	\pi^{\sigma}\begin{pmatrix} a & \mathbf{b}^t\\\mathbf{c} & D \end{pmatrix} f(\bz) 
	= \left(a+\bc^t\bz\right)^{-\sigma} f\left(\frac{\bb+D^t\bz}{{a+\bc^t\bz}}\right),
\end{equation}
defines a family of unitary representation of $\SU(1,d)$ on $\cA_\alpha$ labeled by $\sigma \in \N_{\geq d+1}$.

\section{Matrix coefficients and multivariate Meixner polynomials} \label{sec:Meixner polynomials}
In \cite{Gr75} and \cite{Il12}, the multivariate Meixner polynomials are defined through their generating function. The generating function is given by
\[
G(\bx,\bt,U,\beta) = \left(1 - \sum_{j=1}^d t_j\right)^{-\beta-|\bx|}
\prod_{i=1}^d\left(1 - \sum_{j=1}^d U_{i,j}t_j\right)^{x_i},\qquad \bx,\bt \in \C^d,
\]
where $\beta \in \C\setminus(-\N_0)$, $U=(U_{i,j}) \in \M_d$, and the principal branch of the power function is used. 
Then the multivariate Meixner polynomials $M_\bn(\bx;U,\beta)$ are defined by
\begin{equation} \label{eq:generating function}
G(\bx,\bt,U,\beta) = \sum_{\bn \in \N_0^d} \frac{(\be)_{|\bn|}}{\bn!} M_{\bn}(\bx;U,\beta) \bt^\bn,
\end{equation}
for $\bz$ sufficiently close to $0$. In \cite{Il12} conditions on the matrix $U$ are imposed, but just for the purpose of defining the multivariate Meixner polynomials these conditions are not needed. These multivariate Meixner polynomials have an explicit  expression as a Gelfand-Aomoto hypergeometric series given by
\begin{equation} \label{eq:hypergeometric Meixner}
M_{\bn}(\bx;U,\beta) = \sum_{(a_{i,j}) \in \M_d(\N_0)} \frac{ \prod_{j=1}^d (-n_j)_{\sum_{i=1}^d a_{i,j}} \prod_{i=1}^d (-x_i)_{\sum_{j=1}^d a_{i,j}} }{ (\beta)_{\sum_{i,j=1}^d a_{i,j} } } \prod_{i,j=1}^d \frac{ (1-U_{i,j})^{a_{i,j}} }{a_{i,j}!},
\end{equation}
where $\M_d(\N_0)$ denotes the subset of $\M_d$ consisting of matrices with entries in $\N_0$.

We will show that the multivariate Meixner polynomials occur as matrix coefficients for the holomorphic discrete series representation of $\SU(1,d)$. In this interpretation, the parameter matrix $U$ depends on a $g \in \SU(1,d)$, which implies conditions for $U$ that are closely related to the conditions imposed in \cite{Il12}, see the discussion at the end of this section. For $g \in \SU(1,d)$, the function $\pi^\sigma(g) e_{\m}$ is holomorphic on $\B_d$, hence it must equal its Taylor series which we can consider as the expansion in the basis $\{e_\m \mid \m \in \N_0^d\}$. We consider the corresponding matrix coefficients $\pi_{\m,\bn}^\sigma(g)$ which are determined by
\begin{equation} \label{eq:matrix coefficients}
	\pi^\sigma(g)e_{\bn}(\bz) = \sum_{ \m \in \N_0^d} \pi^{\sigma}_{\m,\bn}(g)e_{\m}(\bz), \qquad \bz \in \B_d,
\end{equation}
or equivalently
\begin{equation} \label{eq:matrix coefficients inner product}
\pi_{\m,\bn}^\sigma(g) = \langle \pi^\sigma(g)e_\bn,e_\m \rangle.
\end{equation}
\begin{theorem} \label{thm:pi(g)=Meixner pol}
	Let $g = \begin{pmatrix} a & \mathbf{b}^t\\\mathbf{c} & D \end{pmatrix} \in \SU(1,d)$ with $a,b_i,c_i\neq 0$ for $i=1,\ldots,d$, then
	\[
	\pi^\sigma_{\m,\bn}(g) = \sqrt{\frac{(\si)_{|\m|} (\si)_{|\bn|}}{\m!\,\bn!} } (-1)^{|\m|} a^{-\sigma} \tilde \bp^\m \bp^{\bn}M_{\m}(\bn;U,\sigma)
	\]
	with $\bp=\bp(g)=(p_1,\ldots,p_d)$, $\tilde \bp=\tilde \bp(g)=(\tilde p_1,\ldots,\tilde p_d)$ and $U=U(g)=(U_{i,j})$ given by  
	\[
	p_i = \frac{b_i}{a}, \quad \tilde p_i = \frac{c_i}{a},\quad  U_{i,j}=\frac{a D_{j,i}}{b_ic_j}.
	\]
\end{theorem}
Before we prove the theorem let us remark that the identities
\[
|a|^2 - \|\bb\|^2 = |a|^2 - \|\bc\|^2 = 1,
\]
imply
\[
|a|^2 = \frac{1}{1-\sum_{i=1}^d|p_i|^2} = \frac1{1-\sum_{i=1}^d|\tilde p_i|^2}.
\]
It is convenient to write $|a|^{-2} = |p_0|^2=|\tilde p_0|^2$, so that
\[
\sum_{i=0}^d |p_i|^2 = \sum_{i=0}^d |\tilde p_i|^2= 1.
\]
\begin{proof}
Assume $g$ is as given in the theorem, then we can write out the left-hand side of \eqref{eq:matrix coefficients} as follows:
\begin{align*}
	\pi^\sigma(g)\bz^{\bn} &= \left(a + \sum_{j=1}^dc_j z_j\right)^{-\sigma} 
	\prod_{i=1}^d\left(\frac{b_i + \sum_{j=1}^d D_{j,i}z_j}{a+\sum_{j=1}^d c_j z_j}\right)^{n_i}\\
	&= a^{-\sigma-|\bn|}\bb^{\bn} \left(1 - \sum_{i=1}^d\frac{-c_iz_i}{a}\right)^{-\sigma-|\bn|}
	\prod_{i=1}^d\left(1 - \sum_{j=1}^d \frac{a D_{j,i}}{b_i c_j}\left(\frac{-c_jz_j}{a}\right)\right)^{n_i}.
\end{align*}
Comparing with the generating function for Meixner polynomials \eqref{eq:generating function} we see that
\begin{equation*}
	\pi^\sigma(g)\bz^{\bn} = \sum_{\m \in \N_0^d } \frac{(\sigma)_{|\m|}}{\m!}a^{-\sigma-|\m|-|\bn|}\bb^{\bn}(-\bc)^{\m}M_{\m}(\bn;U,\sigma)\bz^{\m},
\end{equation*}
where the parameter matrix $U$ is given by $U_{ij}=\frac{aD_{ji}}{b_ic_j}$. From this the result follows.
\end{proof}
From Theorem \ref{thm:pi(g)=Meixner pol} we immediately obtain several properties of the multivariate Meixner polynomials: 
\begin{theorem} \label{thm:properties Meixner}
	The Meixner polynomials $M_\m(\bn;U,\si)$ from Theorem \ref{thm:pi(g)=Meixner pol} have the following properties.
	\begin{enumerate}[(i)]
	\item  Orthogonality relations:
	\[
	\sum_{\bn\in\N_0^d} \frac{(\sigma)_{|\bn|}}{\bn!}\bp^{\bn}\overline{\bp}^\bn M_{\m}(\bn;U,\sigma) \overline{M_{\m'}(\bn;U,\sigma)} = \delta_{\m,\m'}\frac{\m!\tilde \bp^{-\m} \overline{\tilde \bp}^{-\m}}{(\sigma)_{|\m|}|p_0|^{2\sigma}},
	\]
	\[
	\sum_{\m\in\N_0^d} \frac{(\sigma)_{|\m|}}{\m!}\tilde \bp^{\m}\overline{\tilde \bp}^\m M_{\m}(\bn;U,\sigma) \overline{M_{\m}(\bn';U,\sigma)} = \delta_{\bn,\bn'}\frac{\bn!\bp^{-\bn} \overline{\bp}^{-\bn}}{(\sigma)_{|\bn|}|p_0|^{2\sigma}}.
	\]
		\item Integral representation:
		\[
		\begin{split}
			(-1)^{|\m|}& \tilde\bp^\m M_\m(\bn;U,\sigma) =\\
			& \frac{(\sigma -d)_d}{d!} \int_{\B_d} \left(1 + \sum_{i=1}^d\tilde p_i z_i\right)^{-\sigma-|\bn|}
			\prod_{i=1}^d\left(1 + \sum_{j=1}^d U_{i,j}\tilde p_j z_j \right)^{n_i}\overline{\bz}^\m (1-\|\bz\|^2)^{\sigma - d-1}\, dv.
		\end{split}
		\]
		\item Duality: $M_\m(\bn;U,\si) =  M_\bn(\m;U^t,\si)$.
		\item[(iv)] Sum identity: 
		\[
		\begin{split}
		\left( \frac{a(g_1) a(g_2)}{a(g_1g_2)}\right)^\si &\tilde \bp(g_1g_2)^\m \tilde \bp(g_1)^{-\m}  \bp(g_1g_2)^\bn \bp(g_2)^{-\bn} M_\m(\bn;U(g_1g_2),\si) = \\
		&\sum_{\bk \in \N_0^d} (-1)^{|\bk|}  \bp(g_1)^\bk \tilde\bp(g_2)^{\bk} \frac{ (\si)_{|\bk|}}{\bk!} M_\m(\bk;U(g_1),\si) M_\bk(\bn;U(g_2),\sigma), 
		\end{split}
		\]
		where $g_1,g_2 \in \SU(1,d)$ such that $g_1,g_2,g_1g_2$ satisfy conditions as in Theorem \ref{thm:pi(g)=Meixner pol}.
	\end{enumerate}
\end{theorem}
\begin{proof}
	The integral representation follows directly from Theorem \ref{thm:pi(g)=Meixner pol} and \eqref{eq:matrix coefficients inner product}. The duality property follows from unitarity of $\pi^\si$, which implies that $\pi_{\m,\bn}^\sigma(g)=\overline{\pi_{\bn,\m}^\si(g^{-1})}$, and from $g^{-1} = (\begin{smallmatrix} \overline{a}& -\bc^\dagger \\ -\overline{\bb}& D^\dagger \end{smallmatrix} )$, see \eqref{eq:g inverse}. 
	The identity $\pi^\si(g_1)\pi^\si(g_2)=\pi^\si(g_1g_2)$ leads to
	\[
	\pi_{\m,\bn}^\sigma(g_1g_2) = \sum_{\bk \in \N_0^d} \pi_{\m,\bk}^\sigma(g_1) \pi_{\bk,\bn}^\sigma(g_2).
	\]
	Writing this in terms of the Meixner polynomials gives the sum identity. Taking $g_2^{-1}=g_1=g$ and using $\pi_{\m,\bn}^\sigma(g)=\overline{\pi_{\bn,\m}^\si(g^{-1})}$ we obtain the orthogonality relation
	\[
		\delta_{\m,\bn} =\sum_{\bk\in\N_d} \pi^\si_{\m,\bk}(g) \overline{\pi^\si_{\bn,\bk}(g)}.
	\]
	The first orthogonality relation now follows from Theorem \ref{thm:pi(g)=Meixner pol}, and the second orthogonality relation follows from the first one and the duality property.
\end{proof}

\begin{remark}
	Since the monomials $\{e_\bn\}$ form an orthonormal basis for $\cA_\al$, it follows from Parseval's identity that the set $\{\pi_{\m,\cdot}(g) \mid \m \in \N_0^d\}$ is an orthonormal basis for $\ell^2(\N_0^d)$. As a consequence, the set of Meixner polynomials $\{ M_\m(\,\cdot\,;U,\si) \mid \m \in \N_0^d\}$ is an orthogonal basis for the weighted $L^2$-space $\ell^2(\N_0^d; \frac{(\sigma)_{|\bk|}}{\bk!}\bp^{\bk}\overline{\bp}^\bk)$. This gives another proof of Griffiths' \cite{Gr75} result on completeness of the Meixner polynomials.  
\end{remark}
Next, we obtain an identity for the Meixner polynomials which corresponds to the fact that the tensor product of several representations $\pi^\si$ contains a specific $\pi^{\si'}$ as a subrepresentation. The case $N=2$ in the following theorem corresponds to the Runge-type identity in \cite{Gr75}.
\begin{theorem}
	Let $N \in \N_{\geq 2}$, $\si_1,\ldots,\si_N \in \N_{\geq d+1}$ and $\si=\sum_{i=1}^N\si_i$. Define the linear map $\La:\cA_{\si_1-d-1} \tensor \cdots \tensor \cA_{\si_N-d-1} \to \cA_{\si-d-1}$ on basis elements by
	\[
	\La(e_{\m_1} \tensor \cdots \tensor e_{\m_N}) = C_{\m_1,\ldots,\m_N} e_{\m_1+\ldots+\m_N}, \qquad \m_1,\ldots,\m_N \in \N_0^d,
	\]
	with
	\[
	C_{\m_1,\ldots,\m_N}= \sqrt{ \frac{(\m_1+\ldots+\m_N)!}{ (\si)_{|\m_1+\ldots+\m_N|}}\prod_{i=1}^N \frac{ (\sigma_i)_{|\m_i|} }{ \m_i! } },
	\]
	then $\La$ intertwines $\pi^{\si_1} \tensor \cdots \tensor \pi^{\si_N}$ with $\pi^{\si}$. As a consequence, the Meixner polynomials satisfy
	\[
	\frac{(\si)_{|\m|}}{\m!} M_\m(\bn;U,\si) = \displaystyle \sum_{\substack{\m_1,\ldots,\m_N \in \N_0^d \\ \m_1+\ldots+\m_N = \m}} \prod_{i=1}^N \frac{ (\si_i)_{|\m_i|}}{\m_i!}M_{\m_i}(\bn_i;U,\si_i),
	\]
	where  $\m,\bn_1,\ldots,\bn_N \in \N_0^d$ and $\bn = \sum_{i=1}^N \bn_i$. 
\end{theorem}
\begin{proof}
	We write $e_\m^\si$ for a basis vector of $\cA_{\si-d-1}$. Using $\bz^{\m_1} \cdots \bz^{\m_N} = \bz^\m$, with $\m=\m_1+\ldots+\m_N$, we have
	\[
	\prod_{i=1}^N e_{\m_i}^{\si_i}(\bz) = C_{\m_1,\ldots,\m_N} e_{\m}^\si(\bz).
	\]
	Let $\bn_1,\ldots,\bn_N \in \N_0^d$ such that $\sum_{i=1}^N \bn_i = \bn$. Then, using the expansion $\pi^{\si_i}(g)e_{\bn_i} = \sum_{\m_i} \pi^{\si_i}_{\m_i,\bn_i}(g)e_{\m_i}$, we find
	\[
	\begin{split}
	\La\Big( \pi^{\si_1}(g)e_{\bn_1}^{\si_1} & \tensor \cdots \tensor \pi^{\si_N}(g) e_{\bn_N}^{\si_N} \Big) (\bz) \\
	 & =  \sum_{\m_1,\ldots,\m_N \in \N_0^d}  \pi_{\m_1,\bn_1}^{\si_1}(g) \cdots \pi_{\m_N,\bn_N}^{\si_N}(g)  C_{\m_1,\ldots,\m_N}	 e_{\m_1+\ldots+\m_N}^\si(\bz) \\
	 & = \prod_{i=1}^N \sum_{\m_i\in \N_0^d} \pi_{\m_i,\bn_i}^{\si_i}(g) e_{\m_i}^{\si_i}(\bz) \\
	 & = \prod_{i=1}^N \pi^{\si_i}(g) e_{\bn_i}(\bz) \\
	 & = C_{\bn_1,\ldots,\bn_N} \pi^\si(g) e_\bn^\si(\bz),
	\end{split}
	\] 
	where the last step follows from the obvious identity
	\[
	\begin{split}
		\prod_{i=1}^N  (a+\sum_l c_l z_l)^{-\si_i-|\bn_i|} & \prod_{k}(b_k +\sum_l D_{l,k} z_l)^{n_{i,k}} = \\
		&(a+\sum_l c_l z_l)^{-\si-|\bn|} \prod_k (b_k + \sum_l D_{l,k} z_l)^{\bn_{k}}.
	\end{split}
	\]
	This proves the intertwining property of $\La$. Taking the inner product with $e_\m$ shows that
	\[
	\sum_{\substack{\m_1,\ldots,\m_N \in \N_0^d \\ \m_1+\ldots+\m_N = \m}} C_{\m_1,\ldots,\m_N} \pi_{\m_1,\bn_1}^{\si_1}(g) \cdots \pi_{\m_N,\bn_N}^{\si_N}(g) = C_{\bn_1,\ldots,\bn_N} \pi_{\m,\bn}^\si(g),
	\]
	and then the stated identity for Meixner polynomials follows from Theorem \ref{thm:pi(g)=Meixner pol}.
\end{proof}

We conclude this section by comparing our multivariate Meixner polynomials by the ones studied by Iliev in \cite{Il12}.
In \cite{Il12}, the Meixner polynomials depend on a parameter $\beta$ and on the following parameters:
$\bc=(c_1,\ldots,c_d),\tilde\bc=(\tilde c_1,\ldots,\tilde c_d) \in \C^d$ and $U=(U_{ij}) \in \M_d$. These parameters relate to each other as follows: let $C=\diag{1,-\bc}, \widetilde C= \diag{1,-\tilde\bc}$, $\hat U= (\begin{smallmatrix} 1 & \mathbf{1}^t \\ \mathbf{1}& U \end{smallmatrix}) \in \M_{d+1}$ with $\mathbf 1$ the vector in $\C^d$ with every entry equal to 1, then
\begin{equation} \label{eq:parameter conditions}
	\hat U^t C \hat U \widetilde C= c_0 I_{d+1},
\end{equation}
with $c_0=1-|\bc|$. 
This is slightly different from the conditions on the parameters we use in this paper. 
In our approach, the matrix $U$ together with the matrices $C :=\diag{1,-|p_1|^2,\ldots,-|p_d|^2}$ and $\widetilde C:=\diag{1,-|\tilde p_1|^2,\ldots,-|\tilde p_d|^2}$ are obtained from a matrix $g = \left(\begin{smallmatrix} a& \bb^t\\ \bc & D\end{smallmatrix}\right)\in \SU(1,d)$ as follows: $g = a \widetilde P \hat U^t P$ with $P=\diag{1,\bp} = \diag{1,\frac{\bb}{a}}$, $\widetilde P = \diag{1,\tilde \bp}=\diag{1,\frac{\bc}{a}}$ and $\hat U= (\begin{smallmatrix} 1 & \mathbf{1}^t \\ \mathbf{1}& U \end{smallmatrix})$ with $U_{i,j} = \frac{aD_{j,i}}{b_ic_j}$. Then $C = P^\dagger J P$ and $\widetilde C = \widetilde P^\dagger J \widetilde P$, and the condition $g^\dagger J g= J$ is equivalent to
\[
		\hat U^{\dagger} C \hat U \widetilde C = |p_0|^2I_{d+1},
\]
with $|p_0|^2=|a|^{-2} = 1-\sum_i |p_i|^2$. In particular, the only difference with condition \eqref{eq:parameter conditions} is the use of the complex transpose of $\hat U$ instead of just the transpose. Because of this small difference, the orthogonality relations obtained in this paper, and also the difference equations in the next section, are very similar to the results from \cite{Il12}, but they are only the same in case all parameters are real-valued.

\section{The Lie algebra $\su(1,d)$ and multivariate Meixner polynomials} \label{sec:Lie algebra}
The Lie algebra of $\SU(1,d)$ consists of matrices $X \in \M_{d+1}$ with trace zero such that $X^{\dagger}J = -JX$, where (recall) $J=\diag{1,-1,\dots,-1}$. We denote by $\su(1,d)$ the complexification of the Lie algebra of $\SU(1,d)$, i.e.~$\mathfrak{sl}(d+1,\C)$, equipped with the $*$-structure defined by
\[
X^* = JX^\dagger J.
\]
A basis of $\su(1,d)$ is given by 
\begin{equation} \label{eq:basis su(1,d)}
\cB= \Big\{E_{i,j} \mid i,j=0,\ldots,d,\ i \neq j\Big\}\cup \Big\{H_i = E_{i,i} - \frac{1}{d+1}I_{d+1} \mid i=1,\ldots,d\Big\},
\end{equation}
where $E_{i,j}$ denotes the matrix unit with $(i,j)$-entry $1$ and all other entries $0$. Note that
\begin{equation} \label{eq:*-structure}
\begin{aligned}
H_i^*&=H_i, && i=1,\ldots,d,\\
E_{i,0}^*&= - E_{0,i}, &&i=1,\ldots,d,\\
E_{i,j}^*&=E_{j,i}, &&1 \leq i,j\leq d,\ i\neq j.
\end{aligned}
\end{equation}
The unitary representation $\pi^\sigma$ of $\SU(1,d)$ on $\cA_\alpha$ gives rise to an unbounded $*$-representation of $\su(1,d)$ on $\cA_\alpha$ that we also denote by $\pi^\sigma$. As a dense domain we choose the set of polynomials on $\B_d$. On the basis of monomials $\{e_\bn \mid \bn \in \N_0^d\}$, the basis of $\su(1,d)$ acts as follows.

\begin{lemma}
For $\bn \in \N_0^d$,
\[
\begin{aligned}
	\pi^\sigma(H_i)e_{\bn} &= \Big(\frac{\sigma}{d+1} + n_i\Big)\,e_{\bn},&& i=1,\ldots d,\\
	\pi^\sigma(E_{0,j})e_{\bn} &= \sqrt{(\sigma+|\bn|-1)n_j}\,e_{\bn-\bv_j},&& j=1,\ldots,d,\\
	\pi^\sigma(E_{i,0})e_{\bn} &= -\sqrt{(n_i+1)(\sigma+|\bn|)}\,e_{\bn+\bv_i}, &&i=1,\ldots,d,\\
	\pi^\sigma(E_{i,j})e_{\bn} &= \sqrt{(n_i+1)n_j}\,e_{\bn+\bv_i-\bv_j}, && 1 \leq i,j\leq d,\ i\neq j,
\end{aligned}
\]
where $\bv_i$ is the standard basis vector of $\C^{d+1}$ with $j$\textsuperscript{th} entry 1 and the other entries are 0, and we use the convention $e_{\bn}=0$ if $n_i=-1$ for some $1\leq i \leq d$.
\end{lemma}
\begin{proof}
	This follows directly from computing
	\[
	\frac{d}{dt}\Big|_{t=0} \pi^\si(\exp(tX)) e_\bn(\bz), \qquad X \in \cB. \qedhere
	\]
\end{proof}

We fix a $g \in \SU(1,d)$ as in Theorem \ref{thm:pi(g)=Meixner pol}, i.e.
\[
g = \begin{pmatrix} a & \mathbf{b}^t\\\mathbf{c} & D \end{pmatrix}, \qquad \text{with}\ a,b_i,c_i\neq 0\ \text{for}\ i=1,\ldots,d. 
\]
Furthermore, let $U,\bp,\tilde\bp$ be as in Theorem \ref{thm:pi(g)=Meixner pol}. With this $g$, we define a new basis of $\cA_\alpha$ by 
\[
\tilde{e}_{\bn} = \pi^\sigma(g)e_{\bn},\qquad \bn\in\N_0^d,
\]
and define a corresponding basis of $\su(1,d)$ by 
\[
\widetilde \cB = \{ \widetilde{X} =gXg^{-1} \mid X \in \mathcal B\},
\]
where $\mathcal B$ is the basis given in \eqref{eq:basis su(1,d)}. It immediately follows that the action of $\widetilde \cB$ on $\{ \te_\bn \}$ is given by
\begin{equation} \label{eq:tilde action}
\begin{aligned}
	\pi^\sigma(\widetilde H_i)\te_{\bn} &= \Big(\frac{\sigma}{d+1} + n_i\Big)\,\te_{\bn},&& i=1,\ldots d,\\
	\pi^\sigma(\widetilde E_{0,j}) \te_{\bn} &= \sqrt{(\sigma+|\bn|-1)n_j}\,\te_{\bn-\bv_j},&& j=1,\ldots,d,\\
	\pi^\sigma(\widetilde E_{i,0})\te_{\bn} &= -\sqrt{(n_i+1)(\sigma+|\bn|)}\,\te_{\bn+\bv_i}, &&i=1,\ldots,d,\\
	\pi^\sigma(\widetilde E_{i,j})\te_{\bn} &= \sqrt{(n_i+1)n_j}\,\te_{\bn+\bv_i-\bv_j}, && i,j=1,\ldots,d,\ i\neq j.
\end{aligned}
\end{equation}

We will use the representation of $\su(1,d)$ to derive the difference equations for the Meixner polynomials from \cite[Theorem 4.1]{Il12}. First, we need a few preliminary results.
\begin{lemma} \label{lem:Hk} \*
		\begin{enumerate}[(i)]
		\item For $X \in \cB$, we have $(\widetilde X)^*= \widetilde{X^*}$, i.e.
		\[
		(\widetilde H_i)^*=\widetilde H_i, \quad (\widetilde E_{0,i})^*= -\widetilde E_{i,0}, \quad (\widetilde E_{i,j})^*=\widetilde E_{j,i}.
		\]
		for $i,j=1,\ldots,d,\ i\neq j$. 
		
		\item For $k,l=1,\ldots,d$, $k\neq l$, 
		\[
		\begin{split}
		\widetilde H_k & = \sum_{i=1}^d (|D_{i,k}|^2 + |b_k|^2) H_i + \sum^d_{\substack{i,j=1\\ i\neq j}}D_{i,k}\overline{D_{j,k}}E_{i,j} + b_k\sum_{j=1}^d\overline{D_{j,k}}E_{0,j} - \overline{b_k}\sum_{i=1}^dD_{i,k}E_{i,0},\\
		\widetilde E_{k,l} &= \sum_{i=1}^d (D_{i,k}\overline{D_{i,l}} + b_k\overline{b_l}) H_i + \sum^d_{\substack{i,j=1\\ i\neq j}}D_{i,k}\overline{D_{j,l}}E_{i,j} + b_k\sum_{j=1}^d\overline{D_{j,l}}E_{0,j} - \overline{b_l}\sum_{i=1}^dD_{i,k}E_{i,0}, \\
		\widetilde E_{0,l} &= \sum_{i=1}^d (c_i\overline{D_{i,l}} + a\overline{b_l}) H_i + \sum^d_{\substack{i,j=1\\ i\neq j}}c_i \overline{D_{j,l}}E_{i,j} + a\sum_{j=1}^d\overline{D_{j,l}}E_{0,j} - \overline{b_l}\sum_{i=1}^d c_iE_{i,0}, \\
		\widetilde E_{k,0} &= \sum_{i=1}^d (D_{i,k}\overline{c_i} + b_k\overline{a}) H_i + \sum^d_{\substack{i,j=1\\ i\neq j}}D_{i,k}\overline{c_j}E_{i,j} + b_k\sum_{j=1}^d\overline{c_j}E_{0,j} - \overline{a}\sum_{i=1}^dD_{i,k}E_{i,0}.
		\end{split}
		\]
	\end{enumerate}
\end{lemma}
\begin{proof}
	For the first statement we use $\widetilde X = gXg^{-1}$, $J^\dagger=J$ and $I = J^2$, to obtain
	\begin{align*}
		(\widetilde X)^*=J (\widetilde{X})^{\dagger}J &= J (gX g^{-1})^\dagger J=  (Jg^{-1}J)^{\dagger}(JX^\dagger J)(Jg^{\dagger}J).
	\end{align*}	
	Using $g^{-1} = Jg^{\dagger}J$ and $JX^\dagger J=X^*$, it follows that $(\widetilde X)^*=gX^*g^{-1}=\widetilde{X^*}$. 

	The second statement follows from a direct calculation. For $E_{k,l}$ we have
	\[
	\begin{split}
	\widetilde E_{k,l} &= 
	\begin{pmatrix}
			-b_k\overline{b_l} &  b_k\overline{D_{1,l}} & \cdots & b_k \overline{D_{d,l}} \\
			- D_{1,k} \overline{b_l} & D_{1,k}\overline{ D_{1,l}} & \cdots & D_{1,k}  \overline{D_{d,l}} \\
			\vdots & \vdots & \ddots & \vdots\\
			- D_{d,k} \overline{b}_l & D_{d,k} \overline{ D_{1,l}} & \cdots & D_{d,k} \overline{D_{d,l}}
	\end{pmatrix} \\
	&= \sum_{i=1}^d x_i H_i + \sum^d_{\substack{i,j=1\\ i\neq j}}D_{i,k}\overline{D_{j,l}}E_{i,j} + b_k\sum_{j=1}^d\overline{D_{j,l}}E_{0,j} - \overline{b_l}\sum_{i=1}^dD_{i,k}E_{i,0},
\end{split}
\]
where the coefficients $x_i \in \C$ are determined by the equations
\begin{align*}
-\frac{1}{d+1}\sum_{j=1}^d x_j &= -b_k\overline{b_l},\\
x_i -\frac{1}{d+1}\sum_{j=1}^d x_j &= D_{i,k}\overline{D_{i,l}}, \qquad i=1,\ldots,d.
\end{align*}
Note that consistency of these equations follows from the identity $b_k \overline{b_l}+\sum_{i=1}^d D_{i,k} D_{i,l}=0$, see \eqref{eq:identities abcD}. It follows that $x_i = D_{i,k} \overline{D_{i,l}}+b_k\overline{b_l}$.

The results for $\widetilde E_{k,0}$ and $\widetilde E_{0,l}$ follow by interpreting $D_{i,0}$ as $c_i$, and $b_0$ as $a$. The calculation for $\widetilde H_k$ runs along the same lines.
\end{proof}
We are now in a position to derive difference equations for the Meixner polynomials from the action of the Cartan elements $H_k$.
\begin{theorem} \label{thm:difference equations}
For $k=1,\ldots,d$, the Meixner polynomials $M_\m(\bn)=M_\m(\bn;U,\si)$ satisfy
\begin{align*}
	\left| \tfrac{p_0}{p_k} \right|^2 n_k M_{\m}(\bn) 
	=& \left(\sigma + |\m|+\sum_{i=1}^d|U_{k,i}\tilde p_i|^2 m_i\right) M_{\m}(\bn)
	+ \sum^d_{\substack{i,j=1\\ i\neq j}}\overline{U_{k,i}} U_{k,j}|\tilde p_i|^2 m_jM_{\m-\bv_{j}+\bv_{i}}(\bn)\\
	&- \sum_{i=1}^d U_{k,i} m_iM_{\m-\bv_{i}}(\bn)- (\sigma+|\m|)\sum_{i=1}^d \overline{U_{k,i}} |\tilde p_i|^2M_{\m+\bv_{i}}(\bn).
\end{align*}
\end{theorem}
\begin{proof}
The result follows from evaluating $\langle \pi^\si(\widetilde H_k) \te_\bn,e_\m\rangle$ in two ways. 

First note that $\langle \te_\bn, e_\m\rangle = \pi_{\m,\bn}(g)$. From the action of $\widetilde H_k$ \eqref{eq:tilde action}, it follows that
\[
\langle \pi^\sigma(\widetilde H_k) \te_\bn,e_\m\rangle = \left( \frac{\si}{d+1}+n_k\right) \pi_{\m,\bn}(g). 
\]
On the other hand, using $\langle \pi^\sigma(X)\te_\bn,e_\m\rangle = \langle \te_\bn,\pi^\sigma(X^*)e_\m\rangle$ and Lemma \ref{lem:Hk}, we obtain
\begin{align*}
	\langle \pi^\si(\widetilde H_k) \te_\bn, e_\m\rangle =& \sum_{i=1}^d(|D_{i,k}|^2 + |b_k|^2)\Big(\frac{\sigma}{d+1}  + m_i\Big)\pi_{\m,\bn}(g)\\
	+& \sum^d_{\substack{i,j=1\\ i\neq j}}\overline{D_{i,k}}D_{j,k}\sqrt{(m_i+1)m_j}\,\pi_{\m-\bv_j+\bv_i,\bn}(g)\\
	+& \overline{b_k}\sum_{j=1}^dD_{j,k} \sqrt{(\sigma+|\m|-1)m_j} \, \pi_{\m-\bv_j,\bn}(g)\\
	+& b_k\sum_{i=1}^d\overline{D_{i,k}}\sqrt{(m_i+1)(\sigma+|\m|)}\,\pi_{\m+\bv_i,\bn}(g).
\end{align*}
So we have a difference equation for the matrix coefficients $\pi_{\m,\bn}(g)$. Expressing the matrix coefficients in terms of Meixner polynomials using Theorem \ref{thm:pi(g)=Meixner pol}, and simplifying the diagonal terms using the identity $\sum_{i=1}^d |D_{i,k}|^2=|b_k|^2+1$, we obtain a difference equation for the Meixner polynomials. 
\end{proof}

We note that we can rewrite the difference equations as
\begin{align*}
	 n_k M_{\m}(\bn) 
	=&\left|\frac{p_k}{p_0}\right|^2 \sum^d_{\substack{i,j=1\\ i\neq j}}\overline{U_{k,i}} U_{k,j}|\tilde p_i|^2 m_j\Big[M_{\m-\bv_{j}+\bv_{i}}(\bn)-M_\m(\bn)\Big]\\
	&- \left|\frac{ p_k}{p_0}\right|^2\sum_{i=1}^d U_{k,i} m_i\Big[M_{\m-\bv_{i}}(\bn)-M_\m(\bn)\Big]\\
	&- \left|\frac{p_k}{p_0}\right|^2\sum_{i=1}^d \overline{U_{k,i}} |\tilde p_i|^2 (\sigma+|\m|)\Big[M_{\m+\bv_{i}}(\bn) -M_\m(\bn)\Big].
\end{align*}
Comparing this with the difference equations from \cite[Theorem 4.1]{Il12}, we see that the result is again very similar; the difference is the occurrence of complex conjugates of appropriate parameters. \\

In the same way as in Theorem \ref{thm:difference equations} we obtain `lowering and raising' relations for the Meixner polynomials from the actions of $E_{k,l}$. 
\begin{theorem}
	For $k,l=0,\ldots,d$, $k \neq l$,
	\begin{align*}
		\left| \tfrac{p_0}{p_l}\right|^2 n_l M_{\m}&(\bn+\bv_k-\bv_l) = \\
		& \left(\sigma + |\m|+\sum_{i=1}^d \overline{ U_{l,i} } U_{k,i} |\tilde p_i|^2 m_i\right) M_{\m}(\bn)
		+ \sum^d_{\substack{i,j=1\\ i\neq j}} \overline{U_{l,i}} U_{k,j}  |\tilde p_i|^2 m_jM_{\m+\bv_{i}-\bv_{j}}(\bn)\\
		&- \sum_{i=1}^d U_{k,i} m_iM_{\m-\bv_{i}}(\bn)- (\sigma+|\m|)\sum_{i=1}^d \overline{U_{l,i}} |\tilde p_i|^2M_{\m+\bv_{i}}(\bn),
	\end{align*}
	where we use the notations $U_{0,i}=1$, $\bv_0=0$ and $n_0=-\sigma-|\bn|$.
\end{theorem}
Using the duality property of the multivariate Meixner polynomials $M_\m(\bn)$, Theorem \ref{thm:properties Meixner}(iii), it follows that they also satisfy difference equations and lowering/raising relations in the variable $\m$.
	
\section{Degenerate multivariate Meixner polynomials} \label{sec:degenerate Meixner}
So far we considered matrix coefficients $\pi_{\m,\bn}(g)$ where $g = \left(\begin{smallmatrix} a& \bb^t \\ \bc & D \end{smallmatrix} \right) \in \SU(1,d)$ with $a,b_i,c_i \neq 0$. In this section, we consider the degenerate case in which the vectors $\bb$ and $\bc$ may contain elements equal to 0. For convenience we assume 
\[
b_{k+1} = \ldots = b_d =0 \qquad \text{and} \qquad c_{l+1} = \ldots = c_d=0
\]
for some $k,l \in \{1,\ldots,d\}$ and the other elements of $\bb$ and $\bc$ are nonzero. In the non-degenerate case we have an explicit expression for the matrix coefficients using the hypergeometric expression for the multivariate Meixner polynomials. By taking limits we obtain an explicit expression for $\pi_{\m,\bn}^\sigma(g)$.

\begin{lemma} \label{lem:degenerate pi(g)}
	The matrix coefficient $\pi_{\m,\bn}^\sigma(g)$ is given by
	\[
	\begin{split}
		\pi_{\m,\bn}^\sigma(g) &= \sqrt{\frac{(\si)_{|\m|} (\si)_{|\bn|}}{\m!\,\bn!} } (-1)^{|\m|} a^{-\sigma-|\m|-|\bn|} \prod_{i=1}^k \prod_{j=1}^l b_i^{n_i}  c_j^{m_j} \\
		& \quad \times \sum_{(a_{i,j}) \in \M_{k,l}(\m,\bn)} \frac{ \prod_{j=1}^d (-m_j)_{\sum_{i=1}^d a_{i,j}} \prod_{i=1}^d (-n_i)_{\sum_{j=1}^d a_{i,j}}}{ (\si)_{\sum_{i,j=1}^d a_{i,j}}} \prod_{i,j=1}^d \frac{1}{a_{i,j}!} \\
		& \qquad \qquad  \times \prod_{j=1}^l \left(\prod_{i=1}^k  \left( 1- \tfrac{ a D_{j,i}}{b_i c_j}\right)^{a_{i,j}} \prod_{i=k+1}^d \left( - \tfrac{ aD_{j,i} }{c_j}\right)^{a_{i,j}} \right) \\
		&  \qquad \qquad \times  \prod_{j=l+1}^d \left(\prod_{i=1}^k \left( - \tfrac{aD_{j,i}}{b_i} \right)^{a_{i,j}} \prod_{i=k+1}^d (-a D_{j,i})^{a_{i,j}} \right),
	\end{split}
	\]
	where 
	\[
	\begin{split}
		\M_{k,l}(\m,\bn) = \Big\{ (a_{i,j}) \in \M_d(\N_0) \,:& \, \sum_{j=1}^d a_{i,j} = n_i\ \text{for}\ i=k+1,\ldots d \\ &\text{and} \qquad \sum_{i=1}^d a_{i,j} = m_j \ \text{for}\ j=l+1,\ldots,d \Big\}.
	\end{split}
	\]
\end{lemma}
\begin{proof}
	First, we write out the Meixner polynomial $M_\m(\bn;U,\si)$ from Theorem \ref{thm:pi(g)=Meixner pol}, i.e.~in the non-degenerate case, as a hypergeometric series using \eqref{eq:hypergeometric Meixner}, which is a sum labeled by $(a_{i,j}) \in \M_d(\N_0)$ of the form
	\[
	\pi_{\m,\bn}^\sigma(g) = C \left(\prod_{i,j=1}^d b_i^{n_i} c_j^{m_j} \right) \sum_{(a_{i,j})} B_{a_{i,j}} \prod_{i,j=1}^d \left( 1- \frac{aD_{j,i}}{b_ic_j}\right)^{a_{i,j}},
	\]
	where $C$ and $B_{a_{i,j}}$ are independent of $\bb$ and $\bc$. The factor $B_{a_{i,j}}$ contains
	a term $(-n_i)_{\sum_{j=1}^d a_{i,j}}$, which equals $0$ if $\sum_{j=1}^d a_{i,j}>n_i$, so that the sum is a finite sum labeled by $(a_{i,j}) \in \M_d(\N_0)$ with $\sum_{j=1}^d a_{i,j}\leq n_i$ for $1\leq i \leq d$. For taking the limit $b_i \to 0$ we use
	\[
	\lim_{b_i \to 0}  b_i^{n_i} \prod_{j=1}^d \left( 1- \frac{ aD_{j,i} }{b_i c_j}\right)^{a_{i,j}} = 
	\begin{cases}
		\prod_{j=1}^d \left( -\frac{ aD_{j,i} }{c_j} \right)^{a_{i,j}}, & \text{if}\ \sum_{j=1}^d a_{ij} = n_i,\\
		0, & \text{if}\ \sum_{j=1}^d a_{ij} < n_i.
	\end{cases}
	\]
	Applying this for $i=k+1,\ldots,d$, shows that the sum over matrices $(a_{i,j}) \in \M_d(\N_0)$ reduces to a sum over matrices for which the elements of the $i$\textsuperscript{th} row sum to $n_i$, of the form
	\[
	C' \left(\prod_{j=1}^d  c_j^{m_j} \right) \sum_{(a_{i,j})} B_{a_{i,j}} \prod_{j=1}^d \prod_{i=1}^k\left( 1- \frac{ a D_{j,i} }{b_ic_j}\right)^{a_{i,j}} \prod_{i=k+1}^d  \left( -\frac{ aD_{j,i} }{c_j} \right)^{a_{i,j}},
	\]
	where $C'$ is independent of $\bc$.
	$B_{a_{i,j}}$ also contains a term $(-m_j)_{\sum_{i=1}^d a_{i,j}}$, which equals 0 for $\sum_{i=1}^d a_{i,j}> m_j$. Furthermore, we have
	\[
	\begin{split}
		\lim_{c_j \to 0} c_j^{ m_j } & \prod_{i=1}^k\left( 1- \frac{ a D_{j,i} }{b_ic_j}\right)^{a_{i,j}} \prod_{i=k+1}^d \left(- \frac{a D_{j,i} }{c_j} \right)^{a_{i,j}} \\
		& =
		\begin{cases} 
			\prod_{i=1}^k \left( - \frac{aD_{j,i}}{b_i} \right)^{a_{i,j}} \prod_{i=k+1}^d (-a D_{j,i})^{a_{i,j}},& \text{if}\ \sum_{i=1}^d a_{i,j} = m_j,\\
			0, & \text{if}\ \sum_{i=1}^d a_{i,j} < m_j.
		\end{cases}
	\end{split}
	\]
	Applying this for $j=l+1,\ldots,d$, we are left with a sum over matrices $(a_{i,j})$ for which the elements of the $j$\textsuperscript{th} column sum to $m_j$, of the form
	\[
	\begin{split}
		C'' \sum_{(a_{i,j})}  B_{a_{i,j}} &\prod_{j=1}^l \left(\prod_{i=1}^k  \left( 1- \tfrac{ a D_{j,i}}{b_i c_j}\right)^{a_{i,j}} \prod_{i=k+1}^d \left( - \tfrac{ aD_{j,i} }{c_j}\right)^{a_{i,j}} \right) \\
		\times&    \prod_{j=l+1}^d \left(\prod_{i=1}^k \left( - \tfrac{aD_{j,i}}{b_i} \right)^{a_{i,j}} \prod_{i=k+1}^d (-a D_{j,i})^{a_{i,j}} \right).
	\end{split}
	\]
	Writing out $C''$ and $B_{a_{i,j}}$ explicitly gives the result.
\end{proof}
For $\si>0$ and $U \in \M_d$, we define the degenerate multivariate Meixner polynomials $\hat M_\m(\bn;U,\si)$ by
\[
\begin{split}
	\hat M_\m(\bn;U,\si) &= 
	\sum_{(a_{i,j}) \in \M_{k,l}(\m,\bn)} \frac{ \prod_{j=1}^d (-m_j)_{\sum_{i=1}^d a_{i,j}} \prod_{i=1}^d (-n_i)_{\sum_{j=1}^d a_{i,j}}}{ (\si)_{\sum_{i,j=1}^d a_{i,j}}} \prod_{i,j=1}^d \frac{1}{a_{i,j}!} \\
	& \qquad \qquad  \times \prod_{\substack{1 \leq i \leq k\\ 1 \leq j \leq l}}   \left( 1- U_{i,j}\right)^{a_{i,j}} 
	\prod_{\substack{1 \leq i,j \leq d\\ i \geq k+1 \ \text{or} \ j \geq l+1}} \left( - U_{i,j} \right)^{a_{i,j}}.
\end{split}
\]
Define, similar as in Theorem \ref{thm:pi(g)=Meixner pol}, 
\[
\begin{split}
p_i &= \frac{b_i}{a}, \quad i=1,\ldots,k,\\
\tilde p_j &= \frac{ c_i}{a}, \quad j=1,\ldots,l
\end{split}
\]
and $U \in \M_d$  by
\[
U_{i,j} =
\begin{cases} 
	\frac{ a D_{j,i} }{b_i c_j} & 1 \leq i \leq k,\ 1 \leq j \leq l,\\
	\frac{ aD_{j,i}}{b_i} & 1 \leq i \leq k, \ l+1 \leq j \leq d,\\
	\frac{ aD_{j,i}}{c_j} & k+1 \leq i \leq d,\ 1 \leq j \leq l,\\
	a D_{j,i}  & k+1 \leq i \leq d,\ l+1 \leq j \leq d.
\end{cases}
\]
It is convenient to define $\bp,\tilde \bp \in \C^d$ by
\[
\begin{split}
	\bp=(p_1,\ldots,p_k,\tfrac1a,\ldots,\tfrac1a), \qquad 
	\tilde \bp  = (\tilde p_1,\ldots,\tilde p_l,\tfrac1a,\ldots,\tfrac1a).
\end{split}
\]
Note that we now have $D_{j,i} = a p_i \tilde p_j U_{i,j}$ for $1 \leq i,j \leq d$. Then, it follows from Lemma \ref{lem:degenerate pi(g)} that $\pi_{\m,\bn}^\sigma(g)$ is a multiple of $\hat M_{\m}(\bn;U,\si)$,
\[
\pi_{\m,\bn}^\sigma(g) = \sqrt{\frac{(\si)_{|\m|} (\si)_{|\bn|}}{\m!\,\bn!} } (-1)^{|\m|} a^{-\sigma} \tilde{\bp}^{\m}\bp^{\bn}  \hat M_\m(\bn;U,\si).
\]

Properties of the matrix coefficients can easily be translated to properties of the degenerate Meixner polynomials, similar as in the previous sections. We will state the orthogonality relations, the generating function, the duality property and difference equations, and leave the other properties to the interested reader. First we need to introduce some notations, similarly to the notations used at the end of Section \ref{sec:Meixner polynomials}. Given $U \in \M_d$ and $p_i,p_j \in \C$ for $i=1,\ldots,k$, $j=1,\ldots,l$, such that $1-\sum_{i=1}^k |p_i|^2 = 1- \sum_{i=1}^l |\tilde p_i|^2$,  we denote 
\[
\hat U= \begin{pmatrix} 1 & \mathbf{1}^t_l \\ \mathbf{1}_k& U \end{pmatrix},
\]
where $\mathbf{1}_k$ is the vector $\mathbf u$ with $u_i=1$ for $i=1,\ldots,k$ and $u_i=0$ for $i=k+1,\ldots,d$, 
and
\[
\begin{split}
	C &=\diag{1,-|p_1|^2,\ldots,-|p_k|^2,-|p_0|^2,\ldots,-|p_0|^2},\\
	\widetilde C&=\diag{1,-|\tilde p_1|^2,\ldots,-|\tilde p_l|^2,-|p_0|^2,\ldots,-|p_0|^2},
\end{split}
\]
where $|p_0|^2 = 1-\sum_{i=1}^k |p_i|^2$. 
\begin{theorem} Let $\si \in \N_{\geq d+1}$, $p_i,\tilde p_j \in \C$ for $i=1,\ldots,k$, $j=1,\ldots,l$ and $\hat U \in \M_{d}$ such that $\hat U^\dagger C \hat U \widetilde C = |p_0|^2 I_{d+1}$. The polynomials $\hat M_\m(\bn)=\hat M_{\m}(\bn;U,\si)$ have the following properties:
	\begin{enumerate}[(i)]
		\item Orthogonality relations:
		\[
		\begin{split}
			\sum_{\bn \in \N_0^d}   \frac{ (\si)_{|\bn|}}{\bn!}  \bp^\bn \overline{\bp}^\bn \, \hat M_\m(\bn) \overline{\hat M_{\m'}(\bn)}  = \delta_{\m,\m'}\frac{\m!\tilde \bp^{-\m} \overline{\tilde \bp}^{-\m}}{(\sigma)_{|\m|}|p_0|^{2\sigma}}, \\
			\sum_{\m \in \N_0^d}   \frac{ (\si)_{|\m|}}{\m!}  \tilde \bp^{\m} \overline{\tilde \bp}^{\m} \, \hat M_\m(\bn) \overline{\hat M_{\m}(\bn')} = \delta_{\bn,\bn'}\frac{\bn!\tilde \bp^{-\bn} \overline{\tilde \bp}^{-\bn}}{(\sigma)_{|\bn|}|p_0|^{2\sigma}}.
		\end{split}
		\]
		\item Generating function:
		\[
		\begin{split}
			\sum_{\m \in \N_0^d} & \frac{ (\si)_{|\m|}}{\m!} \hat M_\m(\bn;U,\si) \mathbf t^\m \\
			&= \left(1-\sum_{j=1}^l t_j\right)^{-\si-|\bn|} \prod_{i=1}^k \left( 1 - \sum_{j=1}^d U_{i,j} t_j \right)^{n_i} \, \prod_{i=k+1}^d \left(- \sum_{j=1}^d U_{i,j} t_j \right)^{n_i}.
		\end{split}
		\]
		\item Duality: $\hat M_\m(\bn;U,\si) = \hat M_{\bn}(\m;U^t,\si)$.
		\item Difference equations: for $s =1,\ldots,d$,
		\begin{align*}
			\left| \tfrac{p_0}{p_s}  \right|^2 &n_s \hat M_{\m}(\bn) = \\
			& \left(\chi_k(s)(\sigma + |\m|)+\sum_{i=1}^d|U_{s,i}\tilde p_i|^2 m_i\right) \hat M_{\m}(\bn)
			+  \sum^d_{\substack{i,j=1\\ i\neq j}}\overline{U_{s,i}} U_{s,j}|\tilde p_i|^2 m_jM_{\m-\bv_{j}+\bv_{i}}(\bn)\\
			&- \chi_{k}(s) \left(\sum_{i=1}^d U_{s,i} m_i\hat M_{\m-\bv_{i}}(\bn)+ (\sigma+|\m|)\sum_{i=1}^d \overline{U_{s,i}} |\tilde p_i|^2\hat M_{\m+\bv_{i}}(\bn) \right).
		\end{align*}
		where 
		\[
		\chi_{k}(s) = 
		\begin{cases} 
			1, & s \leq k,\\
			0, & s \geq k+1.
		\end{cases}
		\]
	\end{enumerate}
\end{theorem}
\begin{proof}
	The orthogonality relations are orthogonality relations for the matrix coefficients $\pi_{\m,\bn}(g)$. In this case $g = a \widetilde P \hat U^t P$ with $\widetilde P= \diag{1,\tilde \bp}$ and $P = \diag{1,\bp}$, where $p_i,\tilde p_i$ and $U_{i,j}$ are related to $g$ as described above. 
	Note that $\tilde C = \widetilde P^\dagger J \widetilde P$ and $C = P^\dagger J P$. The identity $g^\dagger J g = J$ then leads to the condition $\hat U^\dagger C \hat U \widetilde C = |p_0|^2 I_{d+1}$. The generating function follows from writing out 
	\[
	\pi^\sigma(g) e_\bn(\bz) = \sum_{\m} \pi_{\m,\bn}^\sigma(g) e_{\m}(\bz)
	\]
	in terms of the degenerate Meixner polynomials  and setting 
	\[
	t_i = 
	\begin{cases}
		-\frac{c_i z_i}{a}, & \text{for}\ i=1,\ldots,l,\\
		-\frac{z_i}{a}, & \text{for}\ i=l+1,\ldots,d.
	\end{cases} 
	\]
	The duality property and the difference equations are obtained in the same way as in Theorems \ref{thm:properties Meixner} and \ref{thm:difference equations}.
\end{proof}

\end{document}